\documentclass[11pt,a4paper,fleqn]{amsart}
\usepackage{a4wide,amsfonts,amsmath,latexsym,amssymb,euscript,graphicx,units,mathrsfs}

\usepackage{graphicx}
\usepackage{color}
\usepackage{amssymb}
\usepackage{amssymb}
\usepackage[T1]{fontenc}
\usepackage{latexsym}
\usepackage{xypic}
\usepackage{eufrak}
\usepackage{euscript}
\usepackage{amsfonts,amsmath}
\usepackage{verbatim}
\usepackage{fancyhdr}
\usepackage[english]{babel}
\usepackage{mathrsfs}
\usepackage{units}

\newtheorem{prop}{Proposition}[section]
\newtheorem{theorem*}{Theorem}

\newtheorem{lemma}[prop]{Lemma}
\newtheorem{rem}[prop]{Remark}

\newtheorem{remark}[prop]{Remark}
\newtheorem{theorem}[prop]{Theorem}

\newtheorem*{prop-cont}{Proposition \ref{azerty}}
\newtheorem*{theorem-cont}{Theorem \ref{meckes}}

\renewcommand{\geq}{\geqslant}
\def\leq{\leqslant}

\newcommand{\N}{\mathbb{N}}

\newcommand{\R}{\mathbb{R}}

\def\EE{\mathscr{E}}

\def\1{{\mathbf{1}}}

\def\1{{\mathbf{1}}}
\def\0.5{{\frac{1}{2}}}

\begin{document}

\begin{center}
{\Large{\bf Exchangeable pairs on Wiener chaos}}\\~\\
Ivan Nourdin and Guangqu Zheng \\
{\it Universit\'e du Luxembourg}\\~\\
\end{center}
{\small \noindent {\bf Abstract:} 
In  \cite{nourdinpeccati2009},  Nourdin and Peccati combined the Malliavin calculus  and Stein's method of normal approximation  to associate a rate of convergence to the celebrated fourth moment theorem \cite{FMT} of Nualart and Peccati. Their analysis, known as the Malliavin-Stein method nowadays,   has found many  applications towards stochastic geometry, statistical physics and  zeros of random polynomials, to name a few.  In this article, we further explore the relation between these two fields of mathematics.  In particular, we construct  exchangeable pairs of Brownian motions and we discover a natural link between Malliavin operators and these exchangeable pairs.  By combining our findings with E. Meckes' infinitesimal version of exchangeable pairs, we can give another proof of the quantitative fourth moment theorem. Finally, we extend our result to the multidimensional case.
 \\

\noindent {\bf Key words:} Stein's method;  Exchangeable pairs; Brownian motion; Malliavin calculus.\\

\bigskip

\begin{minipage}[c]{0.04\linewidth}

\end{minipage} \hfill
\begin{minipage}[c]{0.92\linewidth}
{\it Dedicated to the memory of Charles Stein, in remembrance of his beautiful mind and of his inspiring, creative, very original and deep mathematical ideas, which will, for sure, survive him  for a long time.}
\end{minipage}
\begin{minipage}[c]{0.04\linewidth}

\end{minipage} \hfill
\bigskip

\section{Introduction}

At the beginning of the 1970s, Charles Stein, one of the most famous statisticians of the time, introduced in \cite{Stein72} a new revolutionary method for establishing probabilistic approximations (now known as {\it Stein's method}), which is based on the breakthrough application of characterizing differential operators.   The impact of Stein's method and its ramifications during the last 40 years is immense (see for instance the monograph \cite{CGS}), and touches fields as diverse as combinatorics, statistics, concentration and functional inequalities, as well as mathematical physics and random matrix theory.

 Introduced by Paul Malliavin \cite{Malliavin76}, {\it Malliavin calculus} can be roughly described as an infinite-dimensional differential calculus whose operators act on sets of random objects associated with Gaussian or more general noises. 
In 2009, Nourdin and Peccati \cite{nourdinpeccati2009} combined the Malliavin calculus  and Stein's method   for the first time, thus virtually creating a new domain of research, which is now commonly known as the {\it Malliavin-Stein method}. 
  The success of their method relies crucially  on  the existence of integration-by-parts formulae on both sides:  on one side, the Stein's lemma is built on the Gaussian integration-by-parts formula and it is one of the cornerstones of the Stein's method; on the other side, the integration-by-parts formula on Gaussian space is one of the main tools in Malliavin calculus.   Interested readers can refer to the constantly updated website \cite{website} and the monograph \cite{bluebook} for a detailed overview of this active field of research.
  
A prominent example of applying Malliavin-Stein method is the obtention (see also (\ref{nourdinpeccati}) below) of a Berry-Esseen's type rate of convergence associated to the celebrated fourth moment theorem \cite{FMT} of Nualart and Peccati, according to which a standardized sequence of multiple Wiener-It\^o integrals  converges in law to a standard Gaussian random variable if and only if its fourth moment converges to $3$. 

\begin{theorem}\label{np-np}
\begin{enumerate}
\item[(i)] {\rm (Nualart, Peccati \cite{FMT})} Let $(F_n)$ be a sequence of multiple Wiener-It\^o integrals of order $p$, for some fixed $p\geq 1$. Assume that $E[F_n^2]\to \sigma^2>0$ as $n\to\infty$. Then, as $n\to\infty$, we have the following equivalence:
\begin{eqnarray*}  
F_n\overset{\rm law}{\to} N(0,\sigma^2)\quad \Longleftrightarrow\quad E[F_n^4]\to 3\sigma^4.
\end{eqnarray*}
\item[(ii)] {\rm (Nourdin, Peccati \cite{nourdinpeccati2009,bluebook})} Let $F$ be any multiple Wiener-It\^o integral of order $p\geq 1$, such that $E[F^2]= \sigma^2>0$. Then, with  $N\sim N(0,\sigma^2)$ and $d_{TV}$ standing for the total variation distance,
\begin{eqnarray*} 
d_{TV}(F, N)\leq  \frac{2}{\sigma^2}\sqrt{\frac{p-1}{3p}}\, \sqrt{E[F^4]-3\sigma^4}.
\end{eqnarray*}
\end{enumerate}
\end{theorem}

Of course, (ii) was obtained several years after (i), and  (ii) implies `$\Leftarrow$' in (i).
Nualart and Peccati's fourth moment theorem has been the starting point of a number of applications and generalizations by dozens of authors. These collective efforts have allowed one to break several long-standing deadlocks in several domains, ranging from stochastic geometry (see \emph{e.g.} \cite{HLS16, RS2013, Schulte12}) to  statistical physics (see \emph{e.g.} \cite{MP101, MPRW16,MR15}), and zeros of random polynomials (see \emph{e.g.} \cite{ADM16, AL2013, Dalmao15}),  to name a few.  At the time of writing, more than  two hundred papers have been written, which use in one way or the other the Malliavin-Stein method  (see again the webpage  \cite{website}).

Malliavin-Stein method has become a popular tool, especially within the Malliavin calculus community.
Nevertheless, and despite its success, it is less used by researchers who are not specialists of the Malliavin calculus. A possible explanation is that it requires a certain investment before one is in a position to be able to use it, and doing this investment may refrain people who are not originally trained in the Gaussian analysis.
This paper takes its root   from this observation.

During our attempt to make the proof of Theorem \ref{np-np}(ii)  more accessible to readers having no  background on Malliavin calculus , we discover the following interesting fact for exchangeable pairs of multiple Wiener-It\^o integrals. When $p\geq 1$ is an integer and  $f$ belongs to $L^2([0,1]^p)$,
we write $I_p^B(f)$ to indicate the multiple Wiener-It\^o integral of $f$ with respect to Brownian motion $B$, see Section \ref{sec1} for the precise meaning.
\begin{prop}\label{azerty}
Let $(B,B^t)_{t\geq 0}$ be a family of exchangeable pairs of Brownian motions (that is, $B$ is a Brownian motion on $[0,1]$ and, for each $t$, one has $(B,B^t)\overset{\rm law}{=}(B^t,B)$). 
Assume moreover that
\begin{enumerate}
\item[(a)] for any integer $p\geq 1$ and any  $f\in L^2([0,1]^p)$,
 \begin{eqnarray*}
 \lim_{t\downarrow 0} \frac1t \, E\Big[ I_p^{B^t}(f) - I_p^B(f) \big| \sigma\{B\} \Big] =  -p\,I_p^{B}(f)\quad\mbox{in $L^2(\Omega)$}.
 \end{eqnarray*}
  \end{enumerate}
Then, for any integer $p\geq 1$ and any  $f\in L^2([0,1]^p)$,
\begin{enumerate}
      \item[(b)]
        ${\displaystyle \lim_{t\downarrow 0}   \frac{1}{t}\, E\Big[ \big( I_p^{B^t}(f) - I_p^{B}(f) \big)^2 | \sigma\{B\} \Big] =    2  p^2\int_0^1 I^B_{p-1}(f(x,\cdot))^2dx }$ \quad   in $L^2(\Omega)$;
        \item[(c)]
      ${\displaystyle \lim_{t\downarrow 0} \frac{1}{t}\, E\Big[ \big( I_p^{B^t}(f) - I_p^{B}(f) \big)^4  \Big]  =   0}$.
          \end{enumerate}
\end{prop}

Why is this proposition interesting? Because, as it turns out, it combines perfectly well with the following result, which represents the main ingredient from Stein's method we will rely on and which corresponds to a slight modification of a theorem originally due to  Elizabeth Meckes (see \cite[Theorem 2.1]{Meckes06}).
\begin{theorem}[Meckes \cite{Meckes06}]\label{meckes}
Let $F$ and a family of  random variables $(F_t)_{t\geq 0}$   be defined on a common probability space $(\Omega,\mathcal{F},P)$ such that $F_t\overset{law}{=} F$ for every $t\geq 0$. Assume that $F\in L^3(\Omega, \mathscr{G}, P)$ for some $\sigma$-algebra $\mathscr{G}\subset\mathcal{F}$  and that in $L^1(\Omega)$,
\begin{enumerate}
\item[(a)] ${\displaystyle \lim_{t\downarrow 0} \frac1t\,E[F_t-F|\mathscr{G}] = -\lambda\,F}$ for some $\lambda>0$,
\item[(b)] ${\displaystyle   \lim_{t\downarrow 0} \frac1t\,E[(F_t-F)^2|\mathscr{G}] = (2\lambda+S){\rm Var}(F) }$ for some random variable $S$,
\item[(c)] ${\displaystyle  \lim_{t\downarrow 0} \frac1t\,(F_t-F)^3=0}$.
\end{enumerate}
Then, with $N\sim N(0,{\rm Var}(F))$, 
\begin{eqnarray*}
d_{TV}(F,N)\leq \frac{E|S|}{\lambda}.
\end{eqnarray*} 
\end{theorem}

To see how to combine Proposition \ref{azerty} with Theorem \ref{meckes} (see also point(ii) in Remark \ref{A-hyper}), consider indeed a multiple Wiener-It\^o integral of the form $F=I_p^B(f)$, with $\sigma^2=E[F^2]>0$. Assume moreover that 
we have at our disposal a family $\{(B,B^t)\}_{t\geq 0}$ of exchangeable pairs  of Brownian motions, satisfying  the assumption (a) in
Proposition \ref{azerty}. 
Then, putting Proposition \ref{azerty} and Theorem \ref{meckes} together immediately yields that
\begin{eqnarray}\label{nourdinpeccati}
 d_{TV}(F,N)&\leq& \frac{2}{\sigma^2}\,E\left[\left|p\int_0^1 I^B_{p-1}(f(x,\cdot))^2dx-\sigma^2\right|\right].
 \end{eqnarray}
Finally, to obtain the  inequality stated Theorem \ref{np-np}(ii) from (\ref{nourdinpeccati}), it  remains to `play' cleverly with the (elementary) product formula (\ref{product}), see Proposition \ref{proof4} for the details. 

To conclude our elementary proof of Theorem \ref{np-np}(ii), we are thus left 
to construct 
the family $\{(B,B^t)\}_{t>0}$.
Actually, we will offer two constructions with different motivations:
the first one is inspired by Mehler's formula from Gaussian analysis, whereas the second one is more in the spirit of the so-called \emph{Gibbs sampling procedure} within Stein's method  (see \emph{e.g.} \cite[A.2]{Doebler}).

For the first construction, we consider two independent Brownian motions on $[0,1]$ defined on the same probability space $(\Omega,\mathcal{F},P)$, namely $B$ and $\widehat{B}$. We interpolate between  them by considering, for any $t\geq 0$,
\begin{eqnarray*}
B^t = e^{-t}B +\sqrt{1-e^{-2t}}\widehat{B}.
\end{eqnarray*}
It is then easy and straightforward to check that, for any $t\geq 0$,
this new Brownian motion $B^t$, together with $B$, forms an exchangeable pair (see Lemma \ref{exchan-b}).
Moreover, we will compute below (see (\ref{cond-exp-Bett})) that
$E\big[ I_p^{B^t}(f) \big| \sigma\{B\}\big] 
=e^{-pt}\,I_p^{B}(f) $ for any $p\geq 1$ and any $f\in L^2([0,1]^p)$,
from which  (a) in 
Proposition \ref{azerty} immediately follows.
 
 For the second construction,
we consider  two independent Gaussian white noise  $W$ and $W'$ on $[0,1]$ with Lebesgue intensity measure. For each $n\in\N$, we introduce  a uniform partition $\{ \Delta_1, \ldots, \Delta_n \}$ and   a uniformly distributed index $I_n\sim\mathscr{U}_{\{1,\ldots,n\}}$, independent of $W$ and $W'$.  
 For every Borel set $A\subset [0,1]$, we define $W^n(A) = W'(A\cap \Delta_{I_n}) + W(A\setminus \Delta_{I_n})$. This will give us a new Gaussian white noise $W^n$, which will form an exchangeable pair with $W$.   This construction is a particular Gibbs sampling procedure.
 The analogue of (a) in 
Proposition \ref{azerty} is satisfied, namely, if $f\in L^2([0,1]^p)$, $F = I_p^W(f)$ is the $p$th multiple integral with respect to $W$ and $F^{(n)} = I_p^{W^n}(f)$, we have
\begin{eqnarray*}
n E\big[ F^{(n)} - F\big|  \sigma\{W\}\big] \to -pF \quad\mbox{in $L^2(\Omega)$ as $n\to\infty$}.
\end{eqnarray*}
To apply Theorem \ref{meckes} in this setting, we only need to replace $\frac{1}{t}$ by $n$ and replace $F_t$ by $F^{(n)}$.  To get the exchangeable pairs $(B, B^n)$  of Brownian motions in this setting, it suffices to consider  $B(t) = W([0,t])$ and $B^n(t) = W^n([0,t])$, $t\in[0,1]$.  See Section \ref{sec-exch2} for  more precise statements.

Finally, we discuss the extension of our exchangeable pair approach on Wiener chaos to the multidimensional case.
Here again, it works perfectly well, and it allows us to recover the (known) rate of convergence associated with the remarkable Peccati-Tudor theorem \cite{PTudor}. This latter represents a multidimentional counterpart of the fourth moment theorem Theorem \ref{np-np}(i), exhibiting conditions involving only the second and fourth moments that ensure a central limit theorem for random {\it vectors} with chaotic components. 

\begin{theorem}[Peccati, Tudor \cite{PTudor}]\label{PT-thm}
Fix $d\geq 2$ and $p_1,\ldots,p_d\geq 1$.
For each $k\in\{1,\ldots,d\}$, let $(F^k_n)_{n\geq 1}$ be a sequence of multiple Wiener-It\^o integrals of order $p_k$. Assume that $E[F^k_nF^l_n]\to \sigma_{kl}$ as $n\to\infty$ for each pair $(k,l)\in\{1,\ldots,d\}^2$, with $\Sigma=(\sigma_{kl})_{1\leq k,l\leq d}$ non degenerate. Then, as $n\to\infty$, 
\begin{eqnarray}\label{PT}
F_n=(F^1_n,\ldots,F^d_n)\overset{\rm law}{\to} N\sim N(0,\Sigma)\quad \Longleftrightarrow\quad E[(F^k_n)^4]\to 3\sigma_{kk}^2\mbox{ for all $k\in\{1,\ldots,d\}$}.
\end{eqnarray}
\end{theorem}

In \cite{NR}, it is shown that the right-hand side of (\ref{PT}) is also equivalent to
 \begin{eqnarray}\label{NR}
E[\|F_n\|^4]\to E[\|N\|^4]\quad\mbox{as $n\to\infty$},
\end{eqnarray}
where $\|\cdot\|$ stands for the usual Euclidean $\ell^2$-norm of $\R^d$.
Combining the main findings of \cite{NPR} and \cite{NR} yields the following quantitative version associated to Theorem \ref{PT-thm}, which we are able to recover by means of our elementary exchangeable approach.

\begin{theorem}[Nourdin, Peccati, R\'eveillac, Rosi\'nski \cite{NPR,NR}]\label{NPRR-thm}
Let $F=(F^1,\ldots,F^d)$ be a vector composed of multiple Wiener-It\^o integrals $F^k$, $1\leq k\leq d$. Assume that the covariance matrix $\Sigma$ of $F$ is invertible. Then, with $N\sim N(0,\Sigma)$,
\begin{eqnarray}\label{NPRR-stat}
d_W(F,N)\leq \sqrt{d}\|\Sigma\|^{\frac12}_{op}\|\Sigma^{-1}\|_{op}\sqrt{E[\|F\|^4]- E[\|N\|^4]},
\end{eqnarray}
where $d_W$ denotes the Wasserstein distance and $\|\cdot\|_{op}$ the operator norm of a matrix.
\end{theorem}

The currently available proof of (\ref{NPRR-stat}) relies on two main ingredients: 
(i) simple manipulations involving the product formula (\ref{product}) and implying that
\begin{eqnarray*}
\sum_{i,j=1}^d  {\rm Var}\Big(  
p_ip_j\int_0^1 I_{p_i-1}(f_i(x,\cdot))I_{p_j-1}(f_j(x,\cdot))dx
\Big)  \leq E[\|F\|^4]- E[\|N\|^4],
\end{eqnarray*}
(see \cite[Theorem 4.3]{NR} for the details) and (ii) the following inequality shown in \cite[Corollary 3.6]{NPR} by means of the Malliavin operators $D$, $\delta$ and $L$:
\begin{eqnarray}\label{NPR-exch}
d_W(F,N)\leq \sqrt{d}\|\Sigma\|^{\frac12}_{op}\|\Sigma^{-1}\|_{op}
\sqrt{\sum_{i,j=1}^d  {\rm Var}\Big(  
p_ip_j\int_0^1 I_{p_i-1}(f_i(x,\cdot))I_{p_j-1}(f_j(x,\cdot))dx
\Big) }.
\end{eqnarray}
Here, in the spirit of what we have done in dimension one, we also apply our elementary exchangeable pairs approach to prove (\ref{NPR-exch}), with slightly different constants.

The rest of the paper is  organized as follows.  Section 2 contains preliminary knowledge on multiple Wiener-It\^o integrals. In Section 3 (resp. 4), we present 
our first (resp. second) construction of exchangeable pairs of Brownian motions, and we give the main associated properties.
Section 5 is devoted to the proof of Proposition \ref{azerty}, whereas in Section 6 we offer a simple proof of Meckes' Theorem \ref{meckes}. Our new, elementary proof of Theorem \ref{np-np}(ii)  is given in Section 7. In Section 8, we further investigate the connections between our exchangeable pairs   and the  Malliavin operators. Finally, we discuss the extension of our approach to the multidimensional case in Section 9.

\bigskip

{\bf Acknowledgement}.   We would like to warmly thank Christian D\"obler and Giovanni Peccati, for very stimulating discussions on exchangeable pairs since the early stage of this work.

\section{Multiple Wiener-It\^o integrals: definition and elementary properties}\label{sec1} 
In this subsection, we recall  the definition of multiple Wiener-It\^o integrals, and then we give a few {\it soft} properties that will be needed for our new proof of Theorem \ref{np-np}(ii). We refer to the classical monograph \cite{Nualart06} for the details and missing proofs.

Let $f:[0,1]^p\to \R$ be a square-integrable function, with  $p\geq 1$  a given integer.
The $p$th multiple Wiener-It\^o integral of $f$ with respect to the Brownian motion $B=\big(B(x)\big)_{x\in [0,1]}$ is {\it formally} written as
\begin{eqnarray}\label{multiple}
\int_{[0,1]^p} f(x_1,\ldots,x_p) dB(x_1)\ldots dB(x_p).
\end{eqnarray}
To give a precise meaning to (\ref{multiple}), It\^o's crucial idea from the fifties was to first define (\ref{multiple}) for elementary functions
that vanish on diagonals, and then to approximate any $f$ in $L^2([0,1]^p)$ by such elementary functions.

Consider the diagonal set of $[0,1]^p$, that is,
$
D=\{(t_1,\ldots,t_p)\in[0,1]^p:\,\exists i\neq j,\,t_i= t_j\}.
$
Let $\mathcal{E}_p$ be the vector space formed by the set of elementary functions on $[0,1]^p$ that vanish over $D$, that is, the set of those functions $f$ of the form
\begin{eqnarray*}
f(x_1,\ldots,x_p)=\sum_{i_1,\ldots,i_p=1}^k \beta_{i_1\ldots i_p}
{\bf 1}_{
[\tau_{i_1-1},\tau_{i_1})
\times\ldots\times 
[\tau_{i_p-1},\tau_{i_p})
}
(x_1,\ldots,x_p),
\end{eqnarray*} 
where $k\geq 1$ and $0=\tau_0<\tau_1<\ldots<\tau_k$, and the coefficients
$\beta_{i_1\ldots i_p}$ are zero if any two of the indices $i_1,\ldots,i_p$ are equal.
For $f\in \mathcal{E}_p$, we define (without ambiguity with respect to the choice of the representation of $f$)
\begin{eqnarray*}
I^B_p(f) = \sum_{i_1,\ldots,i_p=1}^k \beta_{i_1\ldots i_p}
(B(\tau_{i_1})-B(\tau_{i_1-1}))
\ldots
(B(\tau_{i_p})-B(\tau_{i_p-1})).
\end{eqnarray*}
We also define the symmetrization $\widetilde{f}$ of $f$ by
\begin{eqnarray*}
\widetilde{f}(x_1,\ldots,x_p)=\frac{1}{p!} \sum_{\sigma\in\mathfrak{S}_p}
f(x_{\sigma(1)},\ldots,x_{\sigma(p)}),
\end{eqnarray*} 
where $\mathfrak{S}_p$ stands for the set of all permutations of $\{1,\ldots,p\}$.
The following elementary properties are immediate and easy to prove.
\begin{enumerate}
\item[1.] If $f\in\mathcal{E}_p$, then $I^B_p(f)=I^B_p(\widetilde{f})$.
\item[2.] If $f\in\mathcal{E}_p$ and $g\in\mathcal{E}_q$, then 
$E[I_p^B(f)]=0$
and
$
E[I^B_p(f)I^B_q(g)]=
\left\{
\begin{array}{lll}
0&\mbox{if $p\neq q$}\\
p!\langle \widetilde{f},\widetilde{g}\rangle_{L^2([0,1]^p)}&\mbox{if $p=q$}
\end{array}
\right..
$
\item[3.] The space $\mathcal{E}_p$ is dense in $L^2([0,1]^p)$. In other words, to each $f\in L^2([0,1]^p)$ one can associate a sequence $(f_n)_{n\geq 1}\subset\mathcal{E}_p$ such that $\|f-f_n\|_{L^2([0,1]^p)}\to 0$ as $n\to\infty$.
\item[4.] Since $E[(I_p^B(f_n)-I_p^B(f_m))^2]=p!\|\widetilde{f}_n-\widetilde{f}_m\|_{L^2([0,1]^p)}^2
\leq p!\|f_n-f_m\|_{L^2([0,1]^p)}^2
\to 0$ as $n,m\to\infty$ for $f$ and $(f_n)_{n\geq 1}$ as in the previous point 3, 
we deduce that the sequence $(I_p(f_n))_{n\geq 1}$ is Cauchy in $L^2(\Omega)$ and, as such, it admits a limit denoted by $I_p^B(f)$. It is easy to check that $I_p^B(f)$ only depends on $f$, not on the particular choice of the approximating sequence $(f_n)_{n\geq 1}$, and that points 1 to 3 continue to hold for general $f\in L^2([0,1]^p)$ and $g\in L^2([0,1]^q)$.
\end{enumerate}
We will also crucially rely on the following {\it product formula}, whose proof is elementary and can be made by induction. See, {\it e.g.}, \cite[Proposition 1.1.3]{Nualart06}.
\begin{enumerate}
\item[5.]
For any $p,q\geq 1$, and if $f\in L^2([0,1]^p)$ and $g\in L^2([0,1]^q)$ are symmetric, then
\begin{eqnarray}\label{product}
I_p^B(f) I_q^B(g) = \sum_{r=0}^{p\wedge q} r!\binom{p}{r}\binom{q}{r}I^B_{p+q-2r}(f\otimes_rg),
\end{eqnarray}
where $f\otimes_rg$ stands for the $r$th-contraction of $f$ and $g$, defined as an element of $L^2([0,1]^{p+q-2r})$ by
\begin{eqnarray*}
&&(f\otimes_r g)(x_1,\ldots,x_{p+q-2r}) \\
&=& \int_{[0,1]^r}f(x_1,\ldots,x_{p-r},u_1,\ldots,u_r)g(x_{p-r+1},\ldots,x_{p+q-2r},u_1,\ldots,u_r)
du_1\ldots du_r.
\end{eqnarray*}
\end{enumerate}
Product formula (\ref{product}) has a nice consequence, the inequality
(\ref{hyper1}) below. It is a very particular case of a more general phenomenon satisfied by multiple Wiener-It\^o integrals, the {\it hypercontractivity} property.
\begin{enumerate}
\item[6.]   For any $p\geq 1$, there exists a constant $c_{4,p}>0$ such that, for any (symmetric) $f\in L^2( [0,1]^p)$, 
\begin{eqnarray}\label{hyper1} 
E \big[ I_p^B(f)^4 \big]    \leq c_{4,p}\,  E \big[ I_p^B(f)^2 \big]^2 \, .
\end{eqnarray}
Indeed, thanks to  \eqref{product}  one can write
$
\displaystyle{I_p^B(f)^2 = \sum_{r=0}^p r!\binom{p}{r}^2 I_{2p-2r}^B(f\otimes_r f)}
$
so that
\begin{eqnarray*}
E[I_p^B(f)^4] = \sum_{r=0}^p r!^2\binom{p}{r}^4 (2p-2r)! \|f\widetilde{\otimes}_r f\|_{L^2( [0,1]^{2p-2r})}^2.
\end{eqnarray*}
The conclusion (\ref{hyper1}) follows by observing that 
\begin{eqnarray*}
 p!^2\|f\widetilde{\otimes}_r f\|_{L^2([0,1]^{2p-2r})}^2\leq  p!^2\|f\otimes_r f\|_{L^2( [0,1]^{2p-2r})}^2 \leq p!^2\|f\|_{L^2([0,1]^{p})}^4 = E[I^B_p(f)^2]^2.
 \end{eqnarray*}
 Furthermore, for each $n\geq 2$, using \eqref{product} and induction, one can show that, with $c_{2^n, p}$ a constant depending only on $p$ but not on $f$, 
          $$  E\big[    I_p^B(f)^{2^n} \big]  \leq c_{2^n,p}\,  E \big[ I_p^B(f)^2 \big]^{2^{n-1}} \,. $$
 So for any $r > 2$, there exists an absolute constant $c_{r,p}$ depending only on $p, r$ (but not on $f$) such that 
  \begin{eqnarray}\label{hyper2}  
   E\big[  \vert  I_p^B(f) \vert^r \big]  \leq c_{r,p}\,   E \big[ I_p^B(f)^2 \big]^{r/2} \,\, .
   \end{eqnarray}
 \end{enumerate}

\section{Exchangeable pair of Brownian motions: a first construction}\label{sec-exch}

As anticipated in the introduction,  for this construction we consider two independent Brownian motions on $[0,1]$ defined on the same probability space $(\Omega,\mathcal{F},P)$, namely $B$ and $\widehat{B}$, and we interpolate between them by considering,  for any $t\geq 0$, $B^t = e^{-t}B +\sqrt{1-e^{-2t}}\widehat{B}.$

\begin{lemma}\label{exchan-b}
For each $t\geq 0$, the pair $(B,B^t)$ is exchangeable, that is, 
$(B,B^t)\overset{\rm law}{=}(B^t,B)$.
In particular, $B^t$ is a Brownian motion.
\end{lemma}
\noindent
{\it Proof}.
Clearly, the bi-dimensional process $(B,B^t)$ is Gaussian and centered.
Moreover, for any $x,y\in[0,1]$,
\begin{eqnarray*}
E[B^t(x)B^t(y)]&=&e^{-2t}E[B(x)B(y)]+(1-e^{-2t})E[\widehat{B}(x)\widehat{B}(y)]=
E[B(x)B(y)]\\
E[B(x)B^t(y)] &=& e^{-t} E[B(x)B(y)] = E[B^t(x) B(y)].
\end{eqnarray*}
The desired conclusion follows.\qed

\bigskip

We can now state that, as written in the introduction, our exchangeable pair indeed satisfies the crucial property (a) of
Proposition \ref{azerty}.

\begin{theorem} \label{Bettembourg}
Let $p\geq 1$ be an integer, and consider a kernel $f\in L^2([0,1]^p)$.
Set $F=I_p^B(f)$ and $F_t=I_p^{B^t}(f)$, $t\geq 0$. Then,
\begin{eqnarray}\label{cond-exp-Bett}
E\big[ F_t \big| \sigma\{B\}\big] 
=e^{-pt}\,F.
\end{eqnarray}
In particular, convergence (a) in 
Proposition \ref{azerty} takes place:
 \begin{eqnarray}\label{lili}
 \lim_{t\downarrow 0} \frac1t \, E\Big[ I_p^{B^t}(f) - I_p^B(f) \big| \sigma\{B\} \Big] =  -p\,I_p^{B}(f)\quad\mbox{in $L^2(\Omega)$}.
 \end{eqnarray}
 \end{theorem}
 \noindent
{\it Proof.} 
  Consider first  the case where  $f\in \mathcal{E}_p$, that is, $f$ has the form
\begin{eqnarray*}
f(x_1,\ldots,x_p)=\sum_{i_1,\ldots,i_p=1}^k \beta_{i_1\ldots i_p}
{\bf 1}_{
[\tau_{i_1-1},\tau_{i_1})
\times\ldots\times 
[\tau_{i_p-1},\tau_{i_p})
}
(x_1,\ldots,x_p),
\end{eqnarray*} 
with $k\geq 1$ and $0=\tau_0<\tau_1<\ldots<\tau_k$, and the coefficients
$\beta_{i_1\ldots i_p}$ are zero if any two of the indices $i_1,\ldots,i_p$ are equal. We then have
\begin{eqnarray*}
F_t &=& \sum_{i_1,\ldots,i_p=1}^k \beta_{i_1\ldots i_p}
(B^t(\tau_{i_1})-B^t(\tau_{i_1-1}))
\ldots
(B^t(\tau_{i_p})-B^t(\tau_{i_p-1}))\\
&=& \sum_{i_1,\ldots,i_p=1}^k \beta_{i_1\ldots i_p}
\big[e^{-t}(B(\tau_{i_1})-B(\tau_{i_1-1}))
+\sqrt{1-e^{-2t}}(\widehat{B}(\tau_{i_1})-\widehat{B}(\tau_{i_1-1}))
\big]\\
&&\hskip1.5cm\times\ldots\times
\big[e^{-t}(B(\tau_{i_1})-B(\tau_{i_1-1}))
+\sqrt{1-e^{-2t}}(\widehat{B}(\tau_{i_p})-\widehat{B}(\tau_{i_p-1}))
\big].
\end{eqnarray*}
Expanding and integrating with respect to $\widehat{B}$ yields (\ref{cond-exp-Bett}) for elementary $f$.
Thanks to point 4 in Section \ref{sec1}, we can extend it to any $f\in L^2([0,1]^p)$. We then deduce that
\begin{eqnarray*}\label{thionville}
\frac1t\,E\big[  F_t  -  F \big| \sigma\{B\}\big]  =  \frac{e^{-pt}-1}{t}\,  F ,
\end{eqnarray*}
from which (\ref{lili}) now follows immediately. \qed

\section{Exchangeable pair of Brownian motions: a second construction}\label{sec-exch2}

In this section, we present yet another construction of exchangeable pairs via Gaussian white noise.  We believe it is of independent interest, as such a construction can be similarly carried out for other additive noises.  
This part may be skipped in a first reading, as it is not used in other sections. And we assume that the readers are familiar with the  multiple Wiener-It\^o integrals with respect to the Gaussian white noise, and refer to  \cite[Page 8-13]{Nualart06} for all missing details.


Let $W$ be a   Gaussian white noise on $[0,1]$ with Lebesgue intensity measure $\nu$, 
that is, $W$ is a centred Gaussian process indexed by Borel subsets of $[0,1]$ such that  for any Borel sets $A, B\subset [0,1]$, $W(A)\sim N\big(0, \nu(A)\big)$ and $E\big[ W(A) W(B)\big] = \nu(A\cap B)$.   We  denote by $\mathscr{G}:  =\sigma\{ W\}$ the $\sigma$-algebra generated by  $\big\{ W(A)$: $A$ Borel subset of $[0,1] \big\}$. Now let $W'$ be an independent copy of $W$ (denote by $\mathscr{G}' =\sigma\{W'\}$ the $\sigma$-algebra generated by $W'$)   and   $I_n$ be a uniform random variable over $\{ 1, \ldots, n \}$ for each $n\in\N$ such that $I_n$, $W, W'$ are  independent.   For each fixed $n\in\N$, we  consider the   partition   $[0,1] = \bigcup_{j=1}^n \Delta_j$ 
  with $\Delta_1 = [0, \frac{1}{n}]$, $\Delta_2 = (\frac{1}{n}, \frac{2}{n}]$, $\ldots$ , $\Delta_n = ( 1 -\frac{1}{n} , 1]$.\\

 \paragraph{\bf\small Definition 4.0.}  {\it  Set $ W^n(A) := W'\big(A\cap \Delta_{I_n} \big) +  W\big( A\setminus\Delta_{I_n} \big)$ for any Borel set  $A\subset [0,1]$. }

\begin{remark} {\rm One can first treat $W$  as the superposition of $\big\{ W\vert_{\Delta_j}, j=1, \ldots, n\big\}$, where $W\vert_{\Delta_j}$ denotes the Gaussian white noise on $\Delta_j$.
Then according to  $I_n = j$, we (only) replace $W\vert_{\Delta_j}$ by an independent copy $W'\vert_{\Delta_j}$ so that we get $W^n$.  This is nothing else but a particular  Gibbs sampling procedure (see \cite[A.2]{Doebler}),   hence heuristically speaking, the new process $W^n$ shall form an exchangeable pair with $W$.    
  }

\end{remark}    

  \begin{lemma}\label{verify}  $W$ and $W^n$     form an exchangeable pair with $W$, that is,  $(W, W^n) \overset{law}{=}(W^n, W)$.   In particular, $W^n$ is a Gaussian white noise on $[0,1]$ with  Lebesgue intensity measure.
  
  \end{lemma}
 \begin{proof}  Let us first consider $m$ mutually disjoint Borel sets $A_1, \ldots, A_m\subset [0,1]$. Given $D_1, D_2$   Borel subsets of $\R^m$, we have
  \begin{eqnarray*}
& & P\Big(    \big( W(A_1), \ldots, W(A_m)  \big)\in D_1   \,    ,  \,    \big( W^n(A_1), \ldots, W^n(A_m)  \big)\in D_2  \Big) \\
&=&   \sum_{v = 1}^n  P\Big(      \big( W(A_1), \ldots, W(A_m)  \big)\in D_1   \,    ,  \,    \big( W^n(A_1), \ldots, W^n(A_m)  \big)\in D_2  \,, \, I_n = v \Big)  \\
&=&  \frac{1}{n}\,\, \sum_{v = 1}^n    P\Big( g(X_v, Y_v)\in D_1 , g(X'_v, Y_v)\in D_2    \Big) \,\,,
   \end{eqnarray*}
   where for each $v\in\{ 1, \ldots, n\}$,
\begin{itemize}
\item $X_v:= \big( W(A_1\cap \Delta_v),  \ldots,  W(A_m\cap \Delta_v ) \big)  $,   $X'_v:= \big( W'(A_1\cap \Delta_v),  \ldots,  W'(A_m\cap \Delta_v ) \big)  $,
\item $Y_v:= \big( W(A_1\setminus \Delta_v),  \ldots,  W(A_m\setminus \Delta_v ) \big)  $, and  $g$ is a    function from $\R^{2m}$ to $\R^m$ given by
 $   (x_{1}, \ldots, x_{m} , y_{1}, \ldots, y_{m}   )   \mapsto   g\big(x_{1}, \ldots, x_{m} , y_{1}, \ldots, y_{m}   \big)  = \big( x_1 + y_1, \ldots, x_m+y_m    \big)$
 \end{itemize}
It is clear that for each $v\in\{1, \ldots, n\}$, $X_v, X'_v$ and $Y_v$ are independent, therefore  $g(X_v, Y_v)$ and $g(X'_v, Y_v)$ form an exchangeable pair. It follows from the above equalities that 
  \begin{eqnarray*}
& & P\Big(    \big( W(A_1), \ldots, W(A_m)  \big)\in D_1   \,    ,  \,    \big( W^n(A_1), \ldots, W^n(A_m)  \big)\in D_2  \Big) \\
&=&  \frac{1}{n}\,\, \sum_{v = 1}^n    P\Big( g(X'_v, Y_v)\in D_1 , g(X_v, Y_v)\in D_2    \Big)  \\
&=&  P\Big(    \big( W^n(A_1), \ldots, W^n(A_m)  \big)\in D_1   \,    ,  \,    \big( W(A_1), \ldots, W(A_m)  \big)\in D_2  \Big) \,\, .
   \end{eqnarray*}
 This proves the exchangeability of  $  \big( W(A_1), \ldots, W(A_m)  \big)$ and $\big( W^n(A_1), \ldots, W^n(A_m)  \big)$. 
 
 Now let $B_1, \ldots, B_m$ be Borel subsets of $[0,1]$, then one can find mutually disjoint Borel sets $A_1, \ldots, A_p$ (for some $p\in\N$)  such that each $B_j$ is a union of some of $A_i$'s. Therefore we can find some measurable $\phi: \R^p\to \R^m$ such that $\big( W(B_1), \ldots, W(B_m) \big) = \phi\big( W(A_1), \ldots, W(A_p) \big)$. Accordingly, $\big( W^n(B_1), \ldots, W^n(B_m) \big) = \phi\big( W^n(A_1), \ldots, W^n(A_p) \big)$, hence  $  \big( W(B_1), \ldots, W(B_m)  \big)$ and $\big( W^n(B_1), \ldots, W^n(B_m)  \big)$ are exchangeable.  Now our proof is complete.
  \end{proof}

\begin{remark}\label{BM-rem} {\rm 
For each $t\in[0,1]$, we set $B(t) := W([0,t])$ and $B^n(t) := W^n([0,t])$. Modulo  continuous modifications, one can see from Lemma \ref{verify} that  $B$, $B^n$ are two Brownian motions that  form an exchangeable pair.   An important difference between this construction and the previous one is that $(B, B^t)$ is bi-dimensional Gaussian process whereas $B$, $B^n$ are not   \emph{jointly} Gaussian.

}  
  
\end{remark}

Before we state the analogous result to Theorem \ref{Bettembourg}, we briefly recall the construction of multiple Wiener-It\^o integrals in white noise setting.

\begin{enumerate}
\item
 For each $p\in\N$, we denote by  $\EE_p$ the set of simple functions of the form
 \begin{eqnarray}\label{simple-fct}
 f\big(t_1, \ldots, t_p\big) = \sum_{i_1, \ldots, i_p = 1}^m \beta_{i_1\ldots i_p} \1_{A_{i_1}\times\ldots \times A_{i_p}}\big(t_1, \ldots, t_p\big) \,\, ,
 \end{eqnarray}
where $m\in\N$, $A_1, \ldots, A_m$ are pair-wise disjoint Borel subsets of $[0,1]$, and the coefficients $\beta_{i_1\ldots i_p}$ are zero if any two of the indices  $i_1,\ldots i_p$ are equal. It is known that $\EE_p$ is dense in $L^2([0,1]^p)$.

\item For $f$ given as in \eqref{simple-fct}, the $p$th multiple integral  with respect to $W$ is defined as
 $$I^W_p(f) : = \sum_{i_1, \ldots, i_p = 1}^m \beta_{i_1\ldots i_p} W(A_{i_1}) \ldots  W(A_{i_p} ) \,\, ,  $$
and  one can extend $I_p^W$ to $L^2([0,1]^p)$ via usual approximation argument.  Note $I_p^W(f)$ is nothing else but  $I_p^B(f)$ with the Brownian motion $B$ constructed in Remark \ref{BM-rem}.
  \end{enumerate}

\begin{theorem}\label{GWN}  If   $F= I^W_p(f)$ for some symmetric  $f\in L^2([0,1]^p)$ and we set $F^{(n)} := I_p^{W^n}(f)$,  then   in $L^2(\Omega, \mathscr{G}, P)$ and as $n\to+\infty$,  $n \,   E\big[ F^{(n)} - F \big\vert \mathscr{G} \big] \to  -pF$.

  \end{theorem}
\begin{proof}  First we consider the case where $f\in\mathscr{E}_p$, we assume moreover that $F =\prod_{j=1}^p W(A_j)$ with $A_1, \ldots, A_p$ mutually disjoint Borel subsets of $[0,1]$, and accordingly we define $F^{(n)} =\prod_{j=1}^p W^n(A_j)$. Then, (we write $[p] = \{1, \ldots, p\}$,  $A^v = A\cap \Delta_v$ for any $A\subset [0,1]$ and $v\in\{1, \ldots, n\}$)
\begin{eqnarray*}
 n \,  E\big[ F^{(n)} \big\vert \mathscr{G} \big]     &=& n  \, E\left\{ \sum_{v=1}^n \1_{\{ I_n = v \} } \prod_{j=1}^p \big[ W'(A_j^v) + W(A_j\setminus \Delta_v) \big] \,\, \big\vert \mathscr{G}  \,\,\, \right\}   \\
 &=&  \sum_{v=1}^n\,  E\left\{  \prod_{j=1}^p \big[ W'(A_j^v) + W(A_j\setminus \Delta_v) \big] \,\, \big\vert \mathscr{G}  \,\,\, \right\}  =  \sum_{v=1}^n  \prod_{j=1}^p W\big(  A_j\setminus \Delta_v \big)  \\
 &=&  \sum_{v=1}^n  \Bigg\{ \,\,  \left(  \prod_{j=1}^p   W(  A_j) \right) -  \sum_{k=1}^p W(A_k^v)\left( \prod_{j\in[p]\setminus\{ k\}} W(A_j) \right)  \\
 && \quad + \sum_{\ell=2}^p (-1)^\ell \sum_{\substack{ k_1,\ldots, k_\ell\in [p] \\ \text{all distinct} }} \left(\,\, \prod_{j\in[p] \setminus\{ k_1,\ldots, k_\ell\}} W(A_j) \, \right)  W\big( A^v_{k_1}\big)\cdot\cdot\cdot W\big( A^v_{k_\ell}\big) \,\, \Bigg\}  \\
 &=& n\, F - p\, F +  R_n(F)     \,\,,\\
 \text{where}\quad  R_n(F) &= &\sum_{\ell=2}^p (-1)^\ell \sum_{\substack{ k_1,\ldots, k_\ell\in [p] \\ \text{all distinct} }} \left(\,\, \prod_{j\in[p]\setminus\{ k_1,\ldots, k_\ell\}} W(A_j) \, \right)  \sum_{v=1}^n W\big( A^v_{k_1}\big)\cdot\cdot\cdot W\big( A^v_{k_\ell}\big)   \, .
\end{eqnarray*}
   Then  $R_n(F)$ converges in $L^2(\Omega, \mathscr{G} , P)$ to $0$, due to   the fact that $\sum_{v=1}^n  \prod_{i=1}^q W(A_{k_i}^v)$ converges in $L^2(\Omega)$ to $0$, as $n\to+\infty$,
if $q\geq 2$ and all $k_i$'s are distinct numbers.   This proves our theorem when $f\in\EE_p$.

 By the above computation, we can see that if $F = I^W_p(f)$ with $f$ given in \eqref{simple-fct}, then
$$ R_n(F)  = \sum_{i_1, \ldots, i_p=1}^m \beta_{i_1i_2\ldots i_p}     \sum_{\ell=2}^p (-1)^\ell \sum_{\substack{ k_1,\ldots, k_\ell = 1 \\ \text{all distinct} }}^p \left( \prod_{j\in [p]\setminus\{ k_1,\ldots, k_\ell\}} W(A_{i_j})  \right) \sum_{v=1}^n W\big( A^v_{i_{k_1}}\big)\cdot\cdot\cdot W\big( A^v_{i_{k_\ell}}\big). $$
Therefore, using Wiener-It\^o isometry,  we can first write $ \big\|   R_n(F) \big\| _{L^2(\Omega)}^2$ as
   \begin{eqnarray*}
p! \sum_{i_1, \ldots, i_p=1}^m \big(\beta_{i_1 i_2\ldots i_p}\big)^2  \sum_{v=1}^n   \Bigg\| \,\,  \sum_{\ell=2}^p (-1)^\ell \sum_{\substack{ k_1,\ldots, k_\ell\in[p] \\ \text{all distinct} }} \left( \prod_{j\in[p]\setminus\{ k_1,\ldots, k_\ell\}} W(A_{i_j})  \right) W\big( A^v_{i_{k_1}}\big)\cdot\cdot\cdot W\big( A^v_{i_{k_\ell}}\big) \Bigg\| _{L^2(\Omega)}^2 \, ,
\end{eqnarray*}
and then using the elementary inequality $(a_1 + \ldots + a_m)^\beta \leq m^{\beta - 1} \sum_{i=1}^m \vert a_i\vert^\beta$ for $a_i\in\R$, $\beta > 1$, $m\in\N$, we have 
 \begin{eqnarray*}
& \, &\Bigg\| \,\,  \sum_{\ell=2}^p (-1)^\ell \sum_{\substack{ k_1,\ldots, k_\ell\in[p] \\ \text{all distinct} }} \left( \prod_{j\in[p]\setminus\{ k_1,\ldots, k_\ell\}} W(A_{i_j})  \right) W\big( A^v_{i_{k_1}}\big)\cdot\cdot\cdot W\big( A^v_{i_{k_\ell}}\big) \Bigg\| _{L^2(\Omega)}^2  \\
& \leq&   \Theta_1    \sum_{\ell=2}^p \,\,    \sum_{\substack{ k_1,\ldots, k_\ell\in [p] \\ \text{all distinct} }}   \Bigg\|  \left(\,\, \prod_{j\in  [p] \setminus\{ k_1,\ldots, k_\ell\}} W(A_{i_j}) \, \right)  W\big( A^v_{i_{k_1}}\big)\cdot\cdot\cdot W\big( A^v_{i_{k_\ell}}\big) \Bigg\| _{L^2(\Omega)}^2 \\
& =&    \Theta_1      \sum_{\ell=2}^p    \sum_{\substack{ k_1,\ldots, k_\ell\in  [p]  \\ \text{all distinct} }}   \left(\,\, \prod_{j\in[p] \setminus\{ k_1,\ldots, k_\ell\}} \nu(A_{i_j}) \, \right)    \nu\big( A^v_{i_{k_1}}\big)\cdot\cdot\cdot \nu\big( A^v_{i_{k_\ell}}\big)  \\
&\leq& \Theta_2       \sum_{\substack{ k_1,  k_2\in [p] \\  k_1 \neq k_2 }}   \left(\,\, \prod_{j\in [p] \setminus\{ k_1, k_2 \}} \nu(A_{i_j}) \, \right) \nu\big( A^v_{i_{k_1}}\big)\, \nu\big( A^v_{i_{k_2}}\big) 
\end{eqnarray*}
 where $\Theta_1, \Theta_2$ (and $\Theta_3$ in the following) are some absolute constants that do not depend on $n$ or $F$.     Note now for $k_1\neq k_2$, 
 $ \sum_{v=1}^n \nu\big( A^v_{i_{k_1}}\big)\cdot\nu\big( A^v_{i_{k_2}}\big) \leq  \nu\big( A_{i_{k_1}}\big)  \sum_{v=1}^n  \nu\big( A^v_{i_{k_2}}\big) =   \nu\big( A_{i_{k_1}}\big)\cdot\nu\big( A_{i_{k_2}}\big)$,
thus, 
\begin{eqnarray*}
\big\|   R_n(F) \big\| _{L^2(\Omega)}^2 &\leq&  p!  \sum_{i_1, \ldots, i_p=1}^m \big(\beta_{i_1 i_2\ldots i_p}\big)^2    \Theta_2     \sum_{\substack{ k_1,  k_2\in [p] \\  k_1 \neq k_2 }}   \left(\,\, \prod_{j\in [p] \setminus\{ k_1, k_2 \}} \nu(A_{i_j}) \, \right)  \nu\big( A_{i_{k_1}}\big) \nu\big( A_{i_{k_2}}\big) \\
&\leq& p!  \sum_{i_1, \ldots, i_p=1}^m \big(\beta_{i_1 i_2\ldots i_p}\big)^2    \Theta_3    \prod_{j\in[p]} \nu(A_{i_j}) = \Theta_3 \cdot  \| F \big\| _{L^2(\Omega)}^2 \,\, .
\end{eqnarray*}
Since $\big\{ I_p^W(f)\,:\, f\in\EE_p \big\}$ is dense in the $p$th Wiener chaos $\mathscr{H}_p$,   $R_n: \mathscr{H}_p\to L^2(\Omega)$ is a bounded linear operator with operator norm $\| R_n \| _\text{op} \leq \sqrt{\Theta_3}$ for each $n\in\N$.  Note the linearity follows from its definition $R_n(F) : = n\,  E\big[ F^{(n)} - F \big\vert\mathscr{G} \big] +p F$, $F\in\mathscr{H}_p$.

Now we define 
 $$ \mathscr{C}_p : = \bigg\{ \,  F\in \mathscr{H}_p \,: \, R_\infty(F) : = \lim_{n\to+\infty} R_n(F) \text{\quad is well defined in $L^2(\Omega)$} \bigg\}  \,\, . $$
It is easy to see that $\mathscr{C}_p$ is a dense linear subspace of $\mathscr{H}_p$ and  for each $f\in\EE_p$, $I_p^W(f)\in \mathscr{C}_p$ and $R_\infty(I_p^W(f)) = 0$. As 
$$\sup_{n\in\N} \| R_n \| _\text{op} \leq \sqrt{\Theta_3} < +\infty \,\, ,$$
$R_\infty$ has a unique extension to $\mathscr{H}_p$ and by density of $\big\{ I_p^W(f)\,:\, f\in\EE_p \big\}$  in $\mathscr{H}_p$, $R_\infty(F) = 0$ for each $F \in\mathscr{H}_p$. In other words, for any $F \in\mathscr{H}_p$,  $ n\,  E\big[ F^{(n)} - F \big\vert\mathscr{G} \big]$ converges in $L^2(\Omega)$ to $-pF$, as $n\to+\infty$.   
 \end{proof}

\section{Proof of Proposition \ref{azerty}}

We now give the proof of Proposition \ref{azerty}, which has been stated in the introduction. We restate it for the convenience of the reader.

\begin{prop-cont}
Let $(B,B^t)_{t\geq 0}$ be a family of exchangeable pairs of Brownian motions (that is, $B$ is a Brownian motion on $[0,1]$ and, for each $t$, one has $(B,B^t)\overset{\rm law}{=}(B^t,B)$). 
Assume moreover that
\begin{enumerate}
\item[(a)] for any integer $p\geq 1$ and any  $f\in L^2([0,1]^p)$,
 \begin{eqnarray*}
 \lim_{t\downarrow 0} \frac1t \, E\Big[ I_p^{B^t}(f) - I_p^B(f) \big| \sigma\{B\} \Big] =  -p\,I_p^{B}(f)\quad\mbox{in $L^2(\Omega)$}.
 \end{eqnarray*}
  \end{enumerate}
Then, for any integer $p\geq 1$ and any  $f\in L^2([0,1]^p)$,
\begin{enumerate}
      \item[(b)]
        ${\displaystyle \lim_{t\downarrow 0}   \frac1t\, E\Big[ \big( I_p^{B^t}(f) - I_p^{B}(f) \big)^2 | \sigma\{B\} \Big] =    2  p^2\int_0^1 I^B_{p-1}(f(x,\cdot))^2dx}$ \quad   in $L^2(\Omega)$;
        \item[(c)]
      ${\displaystyle \lim_{t\downarrow 0} \frac{1}{t}\, E\Big[ \big( I_p^{B^t}(f) - I_p^{B}(f) \big)^4  \Big]  =   0}$.
          \end{enumerate}
\end{prop-cont}  
\noindent
{\it Proof}. 
We first concentrate on the proof of (b). 
Fix $p\geq 1$ and  $f\in L^2([0,1]^p)$, and set 
$F= I_p^{B}(f) $ and $F_t= I_p^{B^t}(f)$.
First, we observe that
\begin{eqnarray*}
\frac{1}{t}\,E\big[(F_t-F)^2\big| \sigma\{B\}\big]
=\frac{1}{t}\,
E\big[F_t^2-F^2\big| \sigma\{B\}\big] - \frac{2}{t} F\,E\big[F_t-F\big| \sigma\{B\}\big].
\end{eqnarray*}
Also, as an immediate consequence of the product formula
(\ref{product}) and the definition of $f\otimes_r f$, we have
\begin{eqnarray*}
 p^2\int_0^1 I_{p-1}^B(f(x,\cdot))^2dx= \sum_{r=1}^{p}rr!\binom{p}{r}^2  I_{2p-2r}^{B}(f\otimes_r f).
\end{eqnarray*}
Given (a) and the previous two identities, in order to prove (b) we are thus left to check that
\begin{eqnarray}\label{toshow}
\lim_{t\downarrow 0} \frac{1}{t}\,
E\big[F_t^2-F^2\big| \sigma\{B\}\big] = -2p\, F^2 + 2\sum_{r=1}^{p}rr!\binom{p}{r}^2  I_{2p-2r}^{B}(f\otimes_{r}f)\quad\mbox{in $L^2(\Omega)$.}
\end{eqnarray}
The product formula (\ref{product}) used for multiple integrals 
with respect to $B^t$
(resp. $B$)
yields
\begin{eqnarray*}
F_t^2 = \sum_{r=0}^{p}r!\binom{p}{r}^2 I_{2p-2r}^{B^t}(f\otimes_r  f)\quad\mbox{\Big(resp. $\displaystyle{
F^2 = \sum_{r=0}^{p}r!\binom{p}{r}^2 I_{2p-2r}^{B}(f\otimes_r  f)\Big)}$}.
\end{eqnarray*}
Hence it follows from (a) that
\begin{eqnarray*}
\frac{1}{t}\,   E\big[F_t^2 - F^2\big| \sigma\{B\}\big] &=&
\sum_{r=0}^{p-1}r!\binom{p}{r}^2 \frac1t\,E[I_{2p-2r}^{B^t}(f\otimes_r  f)-
I_{2p-2r}^{B}(f\otimes_r  f)|\sigma\{B\}]\\
&\longrightarrow&
\sum_{r=0}^{p-1}r!\binom{p}{r}^2 (2r-2p) I_{2p-2r}^{B}(f\otimes_r  f)\\
&=& -2p(F^2-E[F^2]) + 2\sum_{r=1}^{p-1}rr!\binom{p}{r}^2  I_{2p-2r}^{B}(f\otimes_r  f),
\end{eqnarray*}
which is exactly (\ref{toshow}). The proof of (b) is complete.

Let us now turn to the proof of (c).
Fix $p\geq 1$ and $f\in L^2([0,1]^p)$, and 
set $F=I_p^B(f)$ and $F_t=I_p^{B^t}(f)$, $t\geq 0$.
We claim that the pair $(F,F_t)$ is exchangeable for each $t$.
Indeed, thanks to point 4 in Section \ref{sec1}, we first observe that it is enough to check this claim when $f$ belongs to $\mathcal{E}_p$, that is, when $f$ has the form
\begin{eqnarray*}
f(x_1,\ldots,x_p)=\sum_{i_1,\ldots,i_p=1}^k \beta_{i_1\ldots i_p}
{\bf 1}_{
[\tau_{i_1-1},\tau_{i_1})
\times\ldots\times 
[\tau_{i_p-1},\tau_{i_p})
}
(x_1,\ldots,x_p),
\end{eqnarray*} 
with $k\geq 1$ and $0=\tau_0<\tau_1<\ldots<\tau_k$, and the coefficients
$\beta_{i_1\ldots i_p}$ are zero if any two of the indices $i_1,\ldots,i_p$ are equal. But, for such an $f$, one has
\begin{eqnarray*}
F=I^B_p(f) &=& \sum_{i_1,\ldots,i_p=1}^k \beta_{i_1\ldots i_p}
(B(\tau_{i_1})-B(\tau_{i_1-1}))
\ldots
(B(\tau_{i_p})-B(\tau_{i_p-1}))\\
F_t=I^{B^t}_p(f) &=& \sum_{i_1,\ldots,i_p=1}^k \beta_{i_1\ldots i_p}
(B^t(\tau_{i_1})-B^t(\tau_{i_1-1}))
\ldots
(B^t(\tau_{i_p})-B^t(\tau_{i_p-1})),
\end{eqnarray*}
and the exchangeability of $(F,F_t)$ follows immediately from those of $(B,B^t)$.
Since the pair $(F,F_t)$ is exchangeable, we can write
    \begin{eqnarray}
  E\big[ (F_t - F)^4 \big] & =&   E\big[ F_t^4 + F^4 - 4 F_t^3F - 4 F^3 F_t + 6 F_t^2 F^2   \big]  \notag\\
  & = &2 E[ F^4] - 8 E\big[ F^3 F_t \big] +  6 E\big[ F^2 F_t^2    \big]  \quad\text{by exchangeability;} \notag\\
  & = &4 E\big[ F^3(F_t - F) \big]  +6 E\big[ F^2   (F_t - F)^2 \big] \quad\text{after rearrangement;}\notag \\
  & = &4 E\big[ F^3 E[ (F_t - F) \vert \sigma\{ B \} ] \big]  +6 E\big[ F^2 E[   (F_t - F)^2 \vert \sigma\{ B \} ] \big].\notag
    \end{eqnarray}   
     Dividing by $t$ and taking the limit $t\downarrow 0$ into the previous identity, we deduce, thanks to (a)  and (b) as well, that
\begin{eqnarray}\label{chantraine}
\lim_{t\downarrow 0} \frac{1}{t}\, E\Big[ \big( F_t -F \big)^4  \Big]  = -4p E[F^4]+12 p^2 \,E\left[F^2 \int_0^1 I^B_{p-1}(f(x,\cdot))^2dx\right].
\end{eqnarray}
In particular, it appears that the limit of $\frac{1}{t}\, E\big[ ( F_t -F )^4  \big]$ is always the same, irrespective of the choice of our exchangeable pair of Brownian motions $(B,B^t)$ satisfying (a). To compute it, we can then choose the pair $(B,B^t)$ we want, for instance, the pair constructed in Section \ref{sec-exch}. This is why, starting from now and for the rest of the proof, $(B,B^t)$ refers to the pair defined in Section \ref{sec-exch} (which satisfies (a), that is, (\ref{lili})).
What we gain by considering this particular pair is that it satisfies a hypercontractivity-type inequality. More precisely, there exists $c_{p}>0$ (only depending on $p$) such that, for all $t\geq 0$, 
\begin{eqnarray}\label{claim}
E[(F_t-F)^4]\leq c_{p} \,E[(F_t-F)^2]^2.
\end{eqnarray}
Indeed, going back to the definition of multiple Wiener-It\^o integrals
as given in Section \ref{sec1} (first for elementary functions and then by approximation for the general case), we see that $F_t-F$ is a multiple Wiener-It\^o integral of order $p$ with respect to the {\it two-sided} Brownian motion $B=(\overline{B}(s))_{s\in[-1,1]}$, defined as
\begin{eqnarray*}
\overline{B}(s)=B(s){\bf 1}_{[0,1]}(s)+\widehat{B}(-s){\bf 1}_{[-1,0]}(s).
\end{eqnarray*}
But product formula (\ref{product}) is also true for a two-sided Brownian motion, so the claim (\ref{claim}) follows from (\ref{hyper1}) applied to $\overline{B}$.
On the other hand, it follows from (b) that $\frac{1}{t}\,E\big[  (  F_t-F  )^2\big]$ converges to a finite number, as $t\downarrow 0$.
Hence, combining this fact with (\ref{claim}) yields
\begin{eqnarray*}
 \frac{1}{t}\, E\Big[ \big(  F_t  -  F \big)^4  \Big] &\leq  c_{p}\,    t\,    \left(   \dfrac{1}{t}\,E\Big[ \big(  F_t  -  F \big)^2  \Big]\right)^{2} \to 0 \, ,
 \end{eqnarray*}
   as $t\downarrow 0$.                  \qed

\begin{remark}\label{A-hyper} \begin{enumerate}
\item[(i)] {\rm
A byproduct of (\ref{chantraine}) in the previous proof is that
\begin{eqnarray}\label{4thmoment}
\frac13\big(E[F^4]-3\sigma^4\big) = E\left[F^2\left(p\int_0^1 I_{p-1}^B(f(x,\cdot))^2dx -  \sigma^2\right)\right].
\end{eqnarray}
Note \eqref{4thmoment} was originally obtained by chain rule, see \cite[equation (5.2.9)]{bluebook}.
}
\item[(ii)] {\rm 
As a consequence of (c) in Proposition \ref{azerty} , we have $\lim_{t\downarrow 0}\frac{1}{t}\, E\big[   \vert I_p^{B^t}(f)  -  I_p^B(f)  \vert^3  \big] =0$.  Indeed, 
\begin{eqnarray*}
  \frac{1}{t} \, E\Big[   \vert I_p^{B^t}(f)  -  I_p^B(f)  \vert^3  \Big] & \leq & \left(   \frac{1}{t} \,  E\Big[   \big(  I_p^{B^t}(f)  -  I_p^B(f) \big)^2  \Big]  \right)^\frac{1}{2}  \left(  \frac{1}{t}\,  E\Big[   \big(  I_p^{B^t}(f)  -  I_p^B(f) \big)^4  \Big]  \right)^\frac{1}{2}  \\
  & \to & 0 \,\,, \quad\text{as $t\downarrow 0$.}
\end{eqnarray*}
}
\item[(iii)]  {\rm    For any $r > 2$,  in view of \eqref{hyper2} and \eqref{claim},  there exists an absolute constant $c_{r,p}$ depending only on $p, r$ (but not on $f$) such that 
  \begin{eqnarray*} 
   E\big[  \vert  I_p^B(f) - I_p^{B^t}(f) \vert^r \big]  \leq c_{r,p}\,   E \big[\big( I_p^B(f) - I_p^{B^t}(f) \big)^2 \big]^{r/2} \,\, .
   \end{eqnarray*}
Moreover, if $F\in L^2(\Omega, \sigma\{B\}, P)$ admits a \emph{finite} chaos expansion, say, (for some $p\in\N$) $F = E[F] + \sum_{q=1}^p I_q^B(f_q)$,  and we set $F_t = E[F] + \sum_{q=1}^p I_q^{B^t}(f_q)$, then there exists some  absolute constant $C_{r,p}$ that only depends on $p$ and $r$ such that 
   $$ E\big[  \vert  F - F_t \vert^r \big]  \leq C_{r,p}\,   E \big[\big( F - F_t \big)^2 \big]^{r/2} \,\, .$$

}
\end{enumerate}
\end{remark}
  
  \section{Proof of E. Meckes' Theorem \ref{meckes}}

In this section, for sake of completeness and because our version slightly differs from the original one given in \cite[Theorem 2.1]{Meckes06}, we provide a proof of Theorem \ref{meckes}, which we restate here for convenience.

\begin{theorem-cont}[Meckes \cite{Meckes06}]
Let $F$ and a family of  random variables $(F_t)_{t\geq 0}$   be defined on a common probability space $(\Omega,\mathcal{F},P)$ such that $F_t\overset{law}{=} F$ for every $t\geq 0$. Assume that $F\in L^3(\Omega, \mathscr{G}, P)$ for some $\sigma$-algebra $\mathscr{G}\subset\mathcal{F}$  and that in $L^1(\Omega)$,
\begin{enumerate}
\item[(a)] ${\displaystyle \lim_{t\downarrow 0} \frac1t\,E[F_t-F|\mathscr{G}] = -\lambda\,F}$ for some $\lambda>0$,
\item[(b)] ${\displaystyle   \lim_{t\downarrow 0} \frac1t\,E[(F_t-F)^2|\mathscr{G}] = (2\lambda+S){\rm Var}(F) }$ for some random variable $S$,
\item[(c)] ${\displaystyle  \lim_{t\downarrow 0} \frac1t\,(F_t-F)^3=0}$.
\end{enumerate}
Then, with $N\sim N(0,{\rm Var}(F))$, 
\begin{eqnarray*}
d_{TV}(F,N)\leq \frac{E|S|}{\lambda}.
\end{eqnarray*} 
\end{theorem-cont}
\noindent
{\it Proof}.
Without loss of generality, we may and will assume that ${\rm Var} (F) = 1$.   It is known that 
\begin{eqnarray} \label{hah}
d_{TV}(F, N)   = \frac{1}{2} \sup   E\big[ \varphi(F) - \varphi(N) \big]  \,\, ,
\end{eqnarray}
where the supremum runs over all smooth functions $\varphi:\R\to\R$ with compact support and such that $\| \varphi\| _\infty \leq 1$. For such a $\varphi$, recall (see, e.g. \cite[Lemma 2.4]{CGS}) that
\begin{eqnarray*}
g(x) = e^{x^2/2} \int_{-\infty}^x  \big( \varphi(y) - E[\varphi(N)] \big) e^{-y^2/2} \, dy \,\, , \quad x\in\R,
\end{eqnarray*}
satisfies 
\begin{eqnarray}\label{steineq}
g'(x) - xg(x) = \varphi(x) - E[ \varphi(N) ]
\end{eqnarray}
as well as $\| g \| _\infty \leq  \sqrt{2\pi} $, $\| g'\| _\infty \leq 4$ and $\| g'' \| _\infty\leq 2 \| \varphi' \| _\infty < +\infty$.
In what follows, we fix such a pair $(\varphi, g)$ of functions.  Let $G$ be a differentiable function such that $G' = g$, then due to $F_t \overset{\rm law}{=} F$,  it follows from the Taylor formula in mean-value form that
\begin{eqnarray*}
 0  =  E\big[ G(F_t) - G(F) \big]   =  E\big[ g(F)(F_t - F) \big] + \frac{1}{2}\,  E\big[ g'(F)(F_t - F)^2 \big] + E[ R]  \,\, ,
 \end{eqnarray*}
 with remainder $R$ bounded by $\frac{1}{6} \| g'' \| _\infty \, \vert F_t - F \vert^3$. 
 
By assumption (c) and as $t\downarrow 0$,
\begin{eqnarray*}
\left\vert \frac{1}{t} \, E [ R   ]\right\vert  \leq  \frac{1}{6} \, \| g'' \| _\infty \, \frac{1}{t} \, E\big[  \vert F_t - F \vert^3  \big] \to 0 .
\end{eqnarray*}
Therefore as $t\downarrow 0$, assumptions (a) and (b) imply that 
 \begin{eqnarray*}
 \lambda \,  E\big[ g'(F) - Fg(F) \big] + \frac{1}{2} \, E\big[ g'(F)S \big] = 0.
  \end{eqnarray*}
Plugging this into Stein's equation (\ref{steineq}) and then using \eqref{hah}, we deduce the desired conclusion, namely, 
 \begin{eqnarray*}
 d_{TV}(F, N) \leq \frac{1}{2} \frac{\| g' \| _\infty}{2\lambda} E| S |  \leq \frac{E|S|}{\lambda}.
 \end{eqnarray*}
\qed

\begin{remark} {\rm  
Unlike the original Meckes' theorem, we do not assume the exchangeability condition  $(F_t, F)\overset{law}{=} (F, F_t)$ in our Theorem \ref{meckes}.  Our consideration is motivated by  \cite{Roellin}. 
 }

\end{remark}

\section{Quantitative fourth moment theorem revisited via exchangeable pairs}

We give an elementary proof to the quantitative fourth moment theorem, that is, we explain how to prove the inequality of Theorem \ref{np-np}(ii) by means of our exchangeable pairs approach.
For sake of convenience, let us restate this inequality: for any multiple Wiener-It\^o integral $F$ of order $p\geq 1$ such that $E[F^2]= \sigma^2>0$, we have, with $N\sim N(0,\sigma^2)$,
\begin{eqnarray}\label{nourdinpeccati-stat}
d_{TV}(F, N)\leq  \frac{2}{\sigma^2}\sqrt{\frac{p-1}{3p}}\, \sqrt{E[F^4]-3\sigma^4}.
\end{eqnarray}

To prove (\ref{nourdinpeccati-stat}), we consider, for instance, the  exchangeable pairs of Brownian motions $\{(B,B^t)\}_{t>0}$  constructed in Section \ref{sec-exch}.
We deduce, by combining Proposition \ref{azerty} with Theorem \ref{meckes} and   Remark \ref{A-hyper}-(ii), that
\begin{eqnarray}\label{nourdinpeccati2}
 d_{TV}(F,N)&\leq& \frac{2}{\sigma^2}\,E\left[\left|p\int_0^1 I^B_{p-1}(f(x,\cdot))^2dx-\sigma^2\right|\right].
 \end{eqnarray}
To deduce (\ref{nourdinpeccati-stat}) from (\ref{nourdinpeccati2}), we are thus left to prove the following result.
\begin{prop}\label{proof4}
Let $p\geq 1$ and consider a symmetric function $f\in L^2([0,1]^p)$. 
Set $F=I_p^B(f)$ and $\sigma^2=E[F^2]$. Then
\begin{eqnarray*}
E\left[\left(p\int_0^1 I^B_{p-1}(f(x,\cdot))^2dx-\sigma^2\right)^2\right]
\leq \frac{p-1}{3p}\big(E[F^4]-3\sigma^4).
\end{eqnarray*}
\end{prop}
\noindent
{\it Proof}. Using the product formula (\ref{product}), we can write
\begin{eqnarray*}
F^2 = \sum_{r=0}^p r!\binom{p}{r}^2 I^B_{2p-2r}(f\otimes_r f) = \sigma^2+\sum_{r=0}^{p-1} r!\binom{p}{r}^2 I^B_{2p-2r}(f\otimes_r f),
\end{eqnarray*}
as well as
\begin{eqnarray*}
&&p\int_0^1 I^B_{p-1}(f(x,\cdot))^2dx=p\sum_{r=0}^{p-1} r!\binom{p-1}{r}^2 I^B_{2p-2r-2}\left(\int_0^1 f(x,\cdot)\otimes_r f(x,\cdot)dx\right)\\
&=&p\sum_{r=1}^{p} (r-1)!\binom{p-1}{r-1}^2 I^B_{2p-2r}\left(f\otimes_r f\right)
=\sigma^2 + \sum_{r=1}^{p-1} \frac{r}{p}\,r!\binom{p}{r}^2 I^B_{2p-2r}\left(f\otimes_r f\right).
\end{eqnarray*}
Hence, by the isometry property (point 2 in Section \ref{sec1}),
\begin{eqnarray*}
E\left[\left(p\int_0^1 I^B_{p-1}(f(x,\cdot))^2dx-\sigma^2\right)^2\right]
=\sum_{r=1}^{p-1} \frac{r^2}{p^2}\,r!^2\binom{p}{r}^4 (2p-2r)! \|f\widetilde{\otimes}_r f\|^2_{L^2([0,1]^{2p-2r})}.
\end{eqnarray*}
On the other hand, one has from (\ref{4thmoment}) and the isometry property again that
\begin{eqnarray*}
\frac13\big(E[F^4]-3\sigma^4\big)&=& E\left[ F^2 \left(p\int_0^1 I^B_{p-1}(f(x,\cdot))^2dx-\sigma^2\right)\right]\\
&=&
\frac13\big(E[F^4]-3\sigma^4\big)=  \sum_{r=1}^{p-1} \frac{r}{p}\,r!^2\binom{p}{r}^4 (2p-2r)! \|f  \widetilde{\otimes}_r f\|^2_{L^2([0,1]^{2p-2r})}.
\end{eqnarray*}
The desired conclusion follows.    \qed

\section{Connections with   Malliavin operators}

Our main goal in this paper is to provide an elementary proof of Theorem \ref{np-np}(ii). Nevertheless, in this section we further investigate the connections we have found between our exchangeable pair approach and the operators of Malliavin calculus.
This part may be skipped in a first reading, as it is not used in other sections.
It is directed to readers who are already familiar with Malliavin calculus. We use classical notation and so do not introduce them in order to save place. We refer to \cite{Nualart06} for any missing detail.

In this section, to stay on the safe side we only consider random variables $F$ belonging to
\begin{eqnarray}\label{AAA}
 \mathcal{A} : = \bigcup_{p\in\N} \bigoplus_{r\leq p} \mathscr{H}_r \,\, ,
\end{eqnarray}
where $\mathscr{H}_r$ is the $r$th chaos associated to the Brownian motion $B$.
In other words, we only consider random variables that are $\sigma\{ B \}$-measurable  and that admit a {\it finite} chaotic expansion.
Note that $\mathcal{A}$ is an algebra (in view of product formula) that is dense 
in $L^2\big(\Omega, \sigma\{B\}, P\big)$.

As is well-known, any $\sigma\{ B \}$-measurable random variable $F$ can be written $F = \psi_F(B)$ for some measurable mapping $\psi_F:\R^{\R_+}\to \R$ determined $P\circ B^{-1}$ almost surely. For such an $F$, we can then define $F_t=\psi_F(B^t)$, with $B^t$ defined in Section \ref{sec-exch}.
Another equivalent description of $F_t$ is to define it as
$
F_t = E[F]+\sum_{r= 1}^p I_r^{B^t}(f_r), 
$
if the family $(f_r)_{1\leq r\leq p}$ is such that $F = E[F]+\sum_{r= 1}^p I_r^{B}(f_r)$.

Our main findings are summarized in the statement below.

\begin{prop} \label{Bettembourg2}
Consider $F,G\in\mathcal{A}$, and define $F_t,G_t$ for each $t\in\R_+$ as is done above.
Then, in $L^2(\Omega)$, 
\begin{enumerate}
\item[(a)] ${\displaystyle \lim_{t\downarrow 0}
       \frac1t \, E\Big[ F_t - F \big| \sigma\{B\} \Big] =  LF}$,
 \item[(b)]  ${ \displaystyle \lim_{t\downarrow 0}
         \frac1t\, E\Big[ \big( F_t - F \big)(G_t - G) | \sigma\{B\} \Big] =   L(FG) - FLG - GLF= 2\, \langle DF, DG \rangle }$.
       \end{enumerate}
\end{prop}
\noindent
{\it Proof.}  The proof of (a) is an immediate consequence of (\ref{lili}), the linearity of conditional expectation, and the fact that $LI_r^B(f_r)=-r\,I_r^B(f_r)$ by definition of $L$.
  Let us now turn to the proof of (b). Using elementary algebra and then (a), we deduce that, as $t\downarrow 0$ and in $L^2(\Omega)$,
\begin{eqnarray*}
&& \frac{1}{t}\,E\big[(F_t-F)(G_t- G)\big| \sigma\{B\}\big]  \\
& =& \frac{1}{t}\, E\big[F_tG_t-FG\big| \sigma\{B\}\big] - \frac{1}{t} F\,E\big[G_t - G\big| \sigma\{W\}\big] - \frac{1}{t} G\, E\big[F_t - F\big| \sigma\{B\}\big]  \\
& \to& L(FG) - FLG - GLF \,\, .
\end{eqnarray*}
Using $L = -\delta D$, $D(FG) = FDG + GDF$ (Leibniz rule) and $\delta (FDG) = F \delta (DG) - \langle DF, DG \rangle$ (see \cite[Proposition 1.3.3]{Nualart06}), it is easy to check that
$L(FG) - FLG - GLF  =  2 \langle DF, DG \rangle$, which concludes the proof of Proposition \ref{Bettembourg2}.\qed

\begin{rem}\label{BELVAL}
{\rm
The expression appearing in the right-hand side of (b) is nothing else but  $2\, \Gamma(F, G)$,  the (doubled) carr\'e du champ operator.  
}
\end{rem}

 To conclude this section, we show how our approach allows to recover  the diffusion property of the Ornstein-Uhlenbeck operator.
 
\begin{prop} Fix $d\in\N$, let $F = (F_1, \ldots, F_d)\in \mathcal{A}^d$ (with $\mathcal{A}$ given in \eqref{AAA}), and $\Psi: \R^d\to \R$ be a polynomial function.  Then
 \begin{eqnarray}
 L \Psi(F) = \sum_{j=1}^d \partial_j \Psi(F) LF_j+ \sum_{i,j=1}^d \partial_{ij}\Psi(F) \langle DF_i, DF_j \rangle\,\, . \label{Diffusion}
 \end{eqnarray}
\end{prop}
 \noindent
 {\it Proof.} We first define $F_t = (F_{1,t}, \ldots, F_{d,t})$ as explained in the beginning of the present section. Using classical multi-index notations, Taylor formula yields that 
\begin{eqnarray}
\Psi(F_t) - \Psi(F) &= &\sum_{j=1}^d \partial_j \Psi(F) \big( F_{j,t} - F_j \big) + \frac{1}{2} \sum_{i,j=1}^d \partial_{i,j} \Psi(F) \big( F_{j,t} - F_j \big) \big( F_{i,t} - F_i \big) \notag \\
& &   + \sum_{\vert\beta\vert = 3}  \frac{3}{\beta_1! \ldots \beta_d!}  (F_t - F)^\beta  \int_0^1 (1-s)^k \big(\partial_1^{\beta_1}\ldots\partial_d^{\beta_d} \Psi\big)\big( F + s(F_t - F)\big) \, ds \, . \label{BBB}
\end{eqnarray} 
In view of the previous proposition, the only difficulty in establishing \eqref{Diffusion} is about controlling the last term in \eqref{BBB} while passing $t\downarrow 0$. Note  first $\big(\partial_1^{\beta_1}\ldots\partial_d^{\beta_d} \Psi\big)\big( F + s(F_t - F)\big)$ is polynomial in $F$ and $(F_t - F)$, so our problem reduces to show 
  \begin{eqnarray}
    \lim_{t\downarrow 0} \frac{1}{t}  E\big[  \vert  F^\alpha   (F_t - F)^\beta \vert \big] = 0 \, , \label{KAWA}
   \end{eqnarray}
for $\alpha = ( \alpha_1, \ldots, \alpha_d), \beta = (\beta_1, \ldots, \beta_d)\in \big( \N \cup \{0\}\big)^d$ with $| \beta | \geq 3$. 

Indeed, (assume $\beta_j > 0$ for each $j$)
\begin{eqnarray*}
  \frac{1}{t}  E\big[  \vert  F^\alpha   (F_t - F)^\beta \vert \big]  & \leq &\frac{1}{t} E\big[  \vert  F^\alpha  \vert^2 \big]^{1/2}  E\big[  \vert    (F_t - F)^\beta \vert^2 \big]^{1/2} \quad\text{by Cauchy-Schwarz inequality;} \\
 &  \leq &  E\big[  \vert  F^\alpha  \vert^2 \big]^{1/2} \, \frac{1}{t} \,    \left(  \prod_{j=1}^d  E\Big[        (F_{j,t} - F_j)^{2\vert \beta \vert} \Big]^{\frac{\beta_j}{ | \beta |}} \right)^{1/2} \quad\text{by H\"older inequality;}  \\
  &  \leq & C \,   E\big[  \vert  F^\alpha  \vert^2 \big]^{1/2} \,  t^{\frac{|\beta|}{2} - 1} \,  \left(  \prod_{j=1}^d \frac{1}{t^{\beta_j}} E\Big[        (F_{j,t} - F_j)^{2} \Big]^{\beta_j} \right)^{1/2} \,\,,
\end{eqnarray*}
where    the last inequality follows from point-(iii) in Remark \ref{A-hyper} with $C > 0$ independent of $t$.   Since $F^\alpha\in\mathcal{A}$ and $| \beta | \geq 3$, \eqref{KAWA} follows  immediately from the above inequalities. 
\qed

\bigskip

\section{Peccati-Tudor theorem revisited too}

In this section, we combine a multivariate version of Meckes' abstract exchangeable pairs \cite{Meckes09} with our results from Section \ref{sec-exch} to prove (\ref{NPR-exch}), thus leading to a fully elementary proof of Theorem \ref{NPRR-thm} as well.

First, we recall the following multivariate version of Meckes' theorem (see  \cite[Theorem 4]{Meckes09}). Unlike in the one-dimensional case,  it seems inevitable to impose the exchangeability condition in the following proposition, as we read from its proof in  \cite{Meckes09}.
 
 \begin{prop}\label{Mec09}
For each $t > 0$, let $(F,F_t)$ be an exchangeable pair of centered  $d$-dimensional random vectors defined on a common probability space. Let $\mathcal{G}$ be a $\sigma$-algebra that contains $\sigma\{F\}$. Assume that $\Lambda\in\R^{d\times d}$ is an invertible deterministic matrix and $\Sigma$ is a symmetric, non-negative definite deterministic matrix such that 
\begin{enumerate}
\item[(a)]
$ {\displaystyle \lim_{t\downarrow 0} \frac1t\,  E\big[ F_t - F |\mathcal{G} \big] =  - \Lambda Y}$ in $L^1(\Omega)$,
\item[(b)]
${\displaystyle
 \lim_{t\downarrow 0} \frac1t\, E\big[ ( F_t - F)(F_t- F)^T |\mathscr{G} \big] =  2\Lambda \Sigma + S }$ in $L^1(\Omega, \| \cdot \| _{\text{HS}})$ for some   matrix $S=S(F)$, and with $\| \cdot \| _{\text{HS}}$ the Hilbert-Schmidt norm
 \item[(c)]
${ \displaystyle \lim_{t\downarrow 0} \sum_{i=1}^d \frac{1}{t}\, E\big[  \vert F_{i,t} - F_i \vert^3 \big] =  0}$, where $F_{i,t}$ (resp. $F_i$) stands for the $i$th coordinate of $F_t$ (resp. $F$).
\end{enumerate}
Then, with $N\sim N_d(0,\Sigma)$, 
\begin{itemize}
\item[(1)] for $g\in C^2(\R^d)$, 
\begin{eqnarray*}
\big| E  [ g(F) ] - E [g(N)  ] \big| \leq  \frac{ \| \Lambda^{-1} \| _{\text{op}}  \sqrt{d}\, M_2(g)  }{4} E \left[\,\,  \sqrt{ \sum_{i,j=1}^d S_{ij}^2 } \,\, \right],
\end{eqnarray*}
where $M_2(g) : = \sup_{x\in\R^d} \big\| D^2 g(x) \big\|_\text{op}$ with $\| \cdot \| _{\text{op}}$ the operator norm.
\item[(2)] if, in addition, $\Sigma$ is positive definite, then 
\begin{eqnarray*}
d_\text{W} (F, N) \leq  \frac{ \| \Lambda^{-1} \| _{\text{op}}  \|  \Sigma^{-1/2} \| _\text{op} }{\sqrt{2\pi}} \,E \left[\,\,  \sqrt{ \sum_{i,j=1}^d S_{ij}^2 } \,\, \right].  \end{eqnarray*}
\end{itemize}
 
 \end{prop}
\begin{remark}{\rm
Constant in (2) is different from Meckes' paper \cite{Meckes09} . We took this better constant from Christian D\"obler's dissertation \cite{Doebler}, see page 114 therein. 
}
\end{remark}

By combining the previous proposition with our exchangeable pairs, we get the following result, whose point 2 corresponds to (\ref{NPR-exch}). 
 
 \begin{theorem}   Fix $d\geq 2$ and  $1\leq p_1 \leq \ldots \leq p_d$. Consider a vector $F: = \big( I^B_{p_1}(f_1), \ldots, I^B_{p_d}(f_d) \big)$ with $f_i\in L^2\big( [0,1]^{p_i}\big)$ symmetric for each $i\in\{1,\ldots,d\}$.
 Let $\Sigma=(\sigma_{ij})$ be the covariance matrix of $F$, and $N \sim N_d(0, \Sigma)$. Then 
 \begin{itemize}
\item[(1)] for $g\in C^2(\R^d)$, 
\begin{eqnarray*}
\Big| E  [ g(F) ] - E [g(N)  ] \Big| \leq  \frac{  \sqrt{d} \,M_2(g)  }{2p_1} \,\sqrt{   \sum_{i,j=1}^d  {\rm Var}\Big(  
p_ip_j\int_0^1 I_{p_i-1}(f_i(x,\cdot))I_{p_j-1}(f_j(x,\cdot))dx
\Big)     },
\end{eqnarray*}
where $M_2(g) : = \sup_{x\in\R^d} \big\| D^2 g(x) \big\|_\text{op}$.
\item[(2)] if in addition, $\Sigma$ is positive definite, then 
\begin{eqnarray*}
d_\text{W} (F, N) \leq  \frac{ 2  \|  \Sigma^{-1/2} \| _\text{op} }{q_1\sqrt{2\pi}} \,\sqrt{   \sum_{i,j=1}^d  {\rm Var}\Big(  
p_ip_j\int_0^1 I_{p_i-1}(f_i(x,\cdot))I_{p_j-1}(f_j(x,\cdot))dx
\Big)     }.
\end{eqnarray*}
 \end{itemize}
 \end{theorem}
 \noindent
{\it Proof}. We consider $F_t =  \big( I^{B^t}_{p_1}(f_1), \ldots, I^{B^t}_{p_d}(f_d) \big)$, where $B^t$ is the Brownian motion constructed in Section \ref{sec-exch}. 
We deduce  from (\ref{cond-exp-Bett})   that 
\begin{eqnarray*}
\frac1t\,  E\big[ F_t - F |\sigma\{B\} \big] 
= \left(\frac{e^{-p_1t}-1}{t}\,I^{B^t}_{p_1}(f_1)
,\ldots,
\frac{e^{-p_dt}-1}{t}\,I^{B^t}_{p_d}(f_d)
\right)
\end{eqnarray*}
implying in turn that, in $L^2(\Omega)$ and as $t\downarrow 0$,
\begin{eqnarray*}
\frac1t\,  E\big[ F_t - F |\sigma\{B\} \big]  \to -\Lambda F,
\end{eqnarray*}
with $\Lambda = \text{diag}(p_1, \ldots, p_d)$ (in particular, $\| \Lambda^{-1} \| _\text{op} = p_1^{-1}$). That is, assumption (a) in Proposition \ref{Mec09} is satisfied (with $\mathcal{G}=\sigma\{B\}$).
That assumption (c) in Proposition \ref{Mec09} is satisfied as well follows from Proposition \ref{azerty}(c).  
Let us finally check that assumption (b)  in Proposition \ref{Mec09}  takes place too. First, using the product formula (\ref{product}) for multiple integrals 
with respect to $B^t$  (resp. $B$)   yields
\begin{eqnarray*}
 F_iF_j  &=& \sum_{r=0}^{p_i\wedge p_j }  r! {p_i\choose r}  {p_j\choose r}  I^B_{p_i+p_j-2r}\big( f_i \otimes_r f_j \big)\\
  F_{i,t}F_{j,t} & =& \sum_{r=0}^{p_i\wedge p_j }  r! {p_i\choose r}  {p_j\choose r}  I^{B^t}_{p_i+p_j-2r}\big( f_i \otimes_r f_j \big).
 \end{eqnarray*}
Hence, using (\ref{lili}) for passing to the limit,
\begin{eqnarray*}
&&\quad \frac1t \, E\big[ (F_{i,t} - F_i )  (F_{j,t}- F_j ) \big| \sigma\{B\} \big] - \frac1t \, E\big[ F_{i,t} F_{j,t}- F_iF_j  \big| \sigma\{B\} \big] \\
& =&   -\frac1t\, F_i \, E\big[ F_{j,t} - F_j | \sigma\{B\}\big]  -  \frac1t\,F_j \, E\big[ F_{i,t} - F_i \big| \sigma\{B\}\big] \\
& \to& (p_i+p_j) F_iF_j =\sum_{r=0}^{p_i\wedge p_j }  r! {p_i\choose r}  {p_j\choose r} (p+q)  I_{p_i+p_j-2r}^B\big( f_i \otimes_r f_j \big) \quad\mbox{as $t\downarrow 0$}.
 \end{eqnarray*}
Now, note  in $L^2(\Omega)$,
\begin{eqnarray*}
& \quad & \frac{1}{t} \, E\big[ F_{i,t}F_{j,t}- F_iF_j  \big| \sigma\{B\} \big] \\
 & =&  \sum_{r=0}^{p_i\wedge p_j }  r! {p_i\choose r}  {p_j\choose r}  \, \frac1t\, E\Big[  I^{B^t}_{p_i+p_j-2r}\big( f_i \otimes_r f_j \big)  - I^B_{p_i+p_j-2r}\big( f_i  \otimes_r f_j \big) \big| \sigma\{B\} \Big] \\
& \to &  \sum_{r=0}^{p_i\wedge p_j }  r! {p_i\choose r}  {p_j\choose r}  \,  (2r - p_i - p_j)  I_{p_i+p_j-2r}^B\big( f_i  \otimes_r f_j \big)  \,\, ,\quad \text{ as $t\downarrow 0$, by (\ref{lili})  .}
 \end{eqnarray*}
Thus, as $t\downarrow 0$,
\begin{eqnarray*}
\frac1t \, E\big[ (F_{i,t} - F_i )  (F_{j,t} - F_j ) \big| \sigma\{B\} \big]  &\to&  2 \sum_{r=1}^{p_i\wedge p_j }  r! r {p_i\choose r}  {p_j\choose r}   I^B_{p_i+p_j-2r}\big( f_i  \otimes_r f_j \big)   \\
&= &2p_ip_j \int_0^1 I_{p_i-1}^B(f_i(x,\cdot))I_{p_j-1}^B(f_j(x,\cdot))dx \, ,
\end{eqnarray*}
where the last equality follows from a straightforward application of the product formula (\ref{product}).  
As a result, if we set  $$
S_{ij} = 2p_ip_j\int_0^1 I_{p_i-1}(f_i(x,\cdot))I_{p_j-1}(f_j(x,\cdot))dx  - 2p_i \sigma_{ij}$$ for each $i,j\in\{1,\ldots,d\}$, then assumption (b) in Proposition \ref{Mec09} turns out to be satisfied as well.  By the isometry property (point 2 in Section \ref{sec1}), it is straightforward to check that
 \begin{eqnarray*}
p_j\int_0^1  E\Big[ I_{p_i-1}(f_i(x,\cdot))I_{p_j-1}(f_j(x,\cdot))\Big] dx  = \sigma_{ij} \,\, .
 \end{eqnarray*}
 Therefore, 
\begin{eqnarray*} 
 E \left[ \,\,  \sqrt{ \sum_{i,j=1}^d S_{ij}^2 } \,\, \right] \leq \sqrt{    \sum_{i,j=1}^d E\big[ S_{ij}^2 \big] } = 2  \sqrt{    \sum_{i,j=1}^d   {\rm Var}\Big(  p_ip_j\int_0^1 I_{p_i-1}(f_i(x,\cdot))I_{p_j-1}(f_j(x,\cdot))dx \Big)  }    \,\, \, .
 \end{eqnarray*} 
Hence the desired results in (1) and (2) follow from Proposition \ref{Mec09}.  \qed

\end{document}